\documentclass[twoside,11pt]{article}
\usepackage{graphicx,amsmath,latexsym,amssymb,amsthm, cite}
\usepackage[colorlinks=black,urlcolor=black]{hyperref}
\usepackage[active]{srcltx}
\usepackage{amsmath}
\usepackage{amsfonts}
\usepackage{amssymb}
\usepackage{amsthm}
\usepackage[T1]{fontenc}
\usepackage{newlfont}
\font\chuto=cmbx10 at 16pt \font\kamy=lcmssb8
at 9pt \font\kam=lcmss8 at 8pt 

\date{}
\newtheorem{theorem}{Theorem}[section]

\newtheorem{lemma}[theorem]{Lemma}

\newtheorem{thm}{Theorem}[section]

\numberwithin{equation}{section}

\begin{document}
	\setlength{\unitlength}{1cm}
	
	\vskip1.5cm
	
	\centerline {\bf \chuto Fractional difference sequence spaces via Pascal mean}

	
	\vskip.8cm \centerline {\kamy  Salila Dutta $^{\dag,}$ and Diptimayee Jena$^\ddag$ }
	
	\vskip.5cm
	\centerline {$^\dag$Department of Mathematics, Utkal University, Vani Vihar} 
	\centerline {Bhubaneswar, India}
	\centerline {e-mail : {\tt saliladutta516@gmail.com}}
	\vskip.1cm
	
	\centerline {$^\ddag$Department of Mathematics, Utkal University, Vanivihar, India} 
	\centerline {e-mail : {\tt jena.deeptimayee@gmail.com }}

	
	\vskip.5cm \noindent{\small{\bf Abstract :} The main purpose 
		of this article is to introduce Pascal difference sequence spaces of fractional order $ \tau $ over the sequence space $\ell_p$ and $\ell_\infty$. Some topological properties of these spaces are considered here along with the Schauder basis,  ${\alpha} -,\beta -$ and $\gamma -$duals of the spaces.\\}

		
		
		\noindent \hrulefill

		\section{Introduction, Preliminaries and Definitions}
		\hskip0.6cm
		Let $ \tau \in \mathbb{R}- \left( {\mathbb{Z}^{-} \bigcup \left\lbrace 0\right\rbrace }\right)  $ where $ \mathbb{R}= $ Set of all real numbers and $ \mathbb{Z}^{-}= $ set of negative integers.	For all real number $ \tau $ ,the Euler gamma function  $\Gamma\left({\tau} \right)$ is defined as 
		\begin{equation}
			\Gamma\left({\tau} \right)=\displaystyle\int_0^\infty e^{-t} \ t^{\tau-1}  \, dt ,
		\end{equation}  which is an improper integral and satisfies the following properties: 
		
		1. $\Gamma\left({n+1} \right)=n!$,  $ n \in \mathbb{N}$ ,( $\mathbb{N} $ is set of natural numbers).
		
		2. $\Gamma\left({\tau+1} \right)= \tau \Gamma\left({\tau} \right)$
		for each real number $\tau \not\in \{0,-1,-2,-3,....\}$.\
		
		\noindent It is obvious by previous knowledge based on sequence spaces the space of all real or complex valued sequences are denoted by $\omega $. Any subspace of $\omega$ is called a sequence
		space. The spaces $c_{0} ,c $ and $l_{\infty }$  are of null, convergent and bounded sequences respectively.\ Also by  $\ell_p$ , we mean the space of absolutely p-summable series, for $ 1 \leq p< \infty $ with the norm $\parallel x \parallel_{\ell_p} = {\left( \sum_{k} {\|x_{k}\|^p}\right)}^{1/p} $, which is a BK- space.  \\
		Let $ M$ and $N$ be two sequence spaces and $ A=\left( {a_{nk}}\right)  $ an infinite matrix of real or comlpex elements. We write  $ Ax = {\left\lbrace A_n (x)\right\rbrace }_{n=0}^{\infty} \in N$ if $ { A_n (x)} = \sum_{k} {a_{nk} x_{k}}$ converges for each $ n $. if  $ x=(x_{k}) \in M $ implies that $ Ax \in N $. Then we say $ A $ defines a matrix mapping from $M$ into $N$ and we denote it be $ A:M \rightarrow N $. And the sequence $ Ax $ is called $ A - $transform of $ x $. 
	The notion of difference sequence spaces $ X\left( \Delta \right)  $ for $ X = \left\lbrace {\ell_\infty} , {c_o} , c \right\rbrace $ was introduced by Kizmaz [\cite{kizm15} ], generalized by Et and \c{C}olak [\cite{et13}] and many others [ see \cite{basa6,madx1, mur1,kada14}]. For a proper fraction ${\tau}$, the concept of fractional difference operator $ \Delta ^{\left( {\tau}\right) } $ was introduced by (Baliarsingh [\cite{bali6}]), Baliarsingh and Dutta [\cite{bali4,bali5,bali7,dutt10,dutt11}] and the difference operator $\Delta ^{\left( {\tau}\right) } $ as
	\begin{equation} \label{1.2} 
		\Delta ^{\left( {\tau}\right)  } x_{k} =\sum _{i}\left(-1\right)^{i} \frac{\Gamma \left({{\tau} }+1\right)}{i!\Gamma ({{\tau} }-i+1)} x_{k-i}   
	\end{equation} 
	
	\begin{equation} \label{1.3} 
		\Delta ^{-\left( {\tau}\right) } x_{k} =\sum _{i}\left(-1\right)^{i} \frac{\Gamma \left({{-\tau} }+1\right)}{i!\Gamma ({-{\tau} }-i+1)} x_{k-i}   
	\end{equation}   
	
	The series of fractional difference operators are convergent. It is also appropriate to express the difference operator $ \left(\Delta ^{\left( {\tau}\right)  }\right)$ and its inverse $ \left(\Delta ^{\left( {-\tau}\right)}\right)  $ respectively, as triangles by :\
	
	\[\left(\Delta ^{\left( {\tau}\right)  } \right)_{nk} =\left\{\begin{array}{l} {\left(-1\right)^{n-k} \frac{\Gamma \left({{\tau} }+1\right)}{\left(n-k\right)!\Gamma ({{\tau} }-n+k+1)} \qquad if\, 0\le k\le n} \\ {0 \qquad \qquad \qquad \qquad \qquad \qquad \qquad if\, k>n} \end{array}\right. \] \
	\begin{equation}\label{1.6}
	\left(\Delta ^{\left( {-\tau}\right)  } \right)_{nk} =\left\{\begin{array}{l} {\left(-1\right)^{n-k} \frac{\Gamma \left({{-\tau} }+1\right)}{\left(n-k\right)!\Gamma ({-{\tau} }-n+k+1)} \qquad if\, 0\le k\le n} \\ {0 \qquad \qquad \qquad \qquad \qquad \qquad \qquad if\, k>n} \end{array}\right. 	
\end{equation}
	 \	The Pascal mean $ P $ [\cite{polat1}] which is described as Pascal matrix $ \left[ p_{nk}\right] $ is defined by \
		\[ {P} = \left[ p_{nk}\right] =\left\{\begin{array}{l} {\left(\begin{array}{l} {n} \\ {n-k} \end{array}\right) \qquad  \qquad if\, 0\le k\le n} \\ {0 \qquad \qquad \qquad \qquad \qquad if\, k>n} \end{array}\right. \] \
		\noindent The inverse of Pascal matrix $ P^{-1} = \left[ p_{nk}\right] ^{-1}  $ is defined by \
\begin{equation}\label{1.7}	
	 P^{-1} =\left\{\begin{array}{l} {\left(-1\right)^{n-k} {\left(\begin{array}{l} {n} \\ {n-k} \end{array}\right)} \qquad if\, 0\le k\le n} \\ {0\qquad \qquad \qquad \qquad  \qquad \qquad if\, k>n} \end{array}\right.  \
\end{equation}	
   Pascal sequence spaces $ {p_\infty} \ , {p_0} \ , and \, {p_c} $ are introduced by Polat[\cite{polat1}] are as follows: \
    \[\begin{array}{l} {p_{0} =\left\{\zeta=\left(\zeta_{k} \right)\in \omega :{\mathop{\lim }\limits_{n\to \infty }} \sum_{k=0}^{n} {\left(\begin{array}{l} {n} \\ {n-k} \end{array}\right)} \zeta_{k} =0 \right\}}, \\ 
   	{p_{c} =\left\{\zeta=\left(\zeta_{k} \right)\in \omega :{\mathop{\lim }\limits_{n\to \infty }} \sum_{k=0}^{n} {\left(\begin{array}{l} {n} \\ {n-k} \end{array}\right)} \zeta_{k} \  exists \right\}}, \\ 
   	{p_{\infty } =\left\{\zeta=\left(\zeta_{k} \right)\in \omega :{\mathop{\sup }\limits_{n}} {\left| \sum_{k=0}^{n} {\left(\begin{array}{l} {n} \\ {n-k} \end{array}\right)}\zeta_{k} \right|} <\infty \right\}} \end{array}\]  \
   Aydin and Polat [\cite{polat3}] introduced Pascal difference sequence spaces of integer order $ m $ as, \
   \begin{equation}\label{1.1}
   \begin{array}{l} {X \left( \Delta^{\left( m \right) } \right)  =\left\{\zeta=\left(\zeta_{k} \right)\in \omega : \Delta^{\left( m \right) } \left( \zeta \right)  \in X \right\}},\\ 
   	\end{array} \
    \end{equation}
    where $  X=\left\lbrace {p_\infty} \ , {p_0}  \ , {p_c} \right\rbrace  $.\\
    Now our interest is to introduce Pascal difference sequence spaces of fractional order. Motivated by equation \ref{1.1} we introduce the spaces $ {\ell_p}{\left( \hat{P}\right) } $, for $ 1 \le p \le \infty $ . We prove certain topological properties of these spaces along with $ \alpha- , \beta- , \gamma- $ duals.\\
	
 \section{Main Results}

\noindent Here we introduce the matrix  $ \hat{P}= {{P} \left( \Delta^{\left( \tau \right)}\right)} $ by taking the product of Pascal matrix ${P}$ [\cite{polat1}] and fractional ordered difference operator  $ \Delta^{\left( \tau \right)}$ \cite{bali4}, as\
\begin{multline} \label{2.1} 
	\hat{P}=\left({{P}\left( \Delta ^{\left( \tau\right)}\right) } \right)_{nk} =\left\{\begin{array}{l} {\sum _{i=k}^{n}(-1)^{i-k} }  \left(\begin{array}{l} \ {n} \\ {n-i} \end{array}\right)\frac{\Gamma \left({\tau} +1\right)}{\left(i-k\right)!\Gamma \left({\tau} -i+k+1\right)} \  , { \qquad if\, 0\le k\le n } \\ {0,\qquad \qquad \qquad \qquad \qquad \qquad\qquad \qquad \qquad \qquad if\, k>n} \end{array}\right. 
\end{multline} 		
	Equivalently, one  may write, \
	\begin{multline}\label{*}
		\left({\bar{P}\left( \Delta ^{\left( \tau\right)}\right) } \right) =\
		\left(%
		\begin{array}{ccccccc}
			{1 } & 0 & 0 & 0 & \dots \\
			({2-\tau })  &  ({1})   & 0 & 0 & \dots  \\
			({3-3\tau + \frac{\tau{\left( \tau-1 \right) }}{2!}})  &  ({3-\tau })  &  ({1})  & 0&\dots  \\
			\vdots  & \vdots  & \vdots  &\vdots &\ddots
		\end{array}%
		\right)
	\end{multline}\\
\begin{thm}\label{thm2.1}
	The inverse of the product matrix $\hat{P}=\left({{P}\left( \Delta ^{\left( \tau\right)}\right) } \right) $ is given by
	$\hat{P}^{-1} =\\
	\left\{\begin{array}{l} { \sum _{j=k}^{n}\left(-1\right)^{n-k}  \left(\begin{array}{l} {j} \\  {j-k} \end{array}\right) \frac{\Gamma \left({-\tau} +1\right)}{\left(n-j\right)!\Gamma \left({-\tau}-n+j+1\right)}  , \qquad \qquad   if\, 0\le k\le n } \\ {0, \qquad \qquad \qquad \qquad \qquad \qquad \qquad \qquad  \qquad  \qquad \qquad  if\, k>n} \end{array}\right. $ 	
	\end{thm}

\begin{proof}
	The proof is a routine verification using results \ref{1.6} and \ref{1.7} and hence omitted.
\end{proof}
\noindent Now we introduce the new Pascal difference sequence spaces of fractional order  $ {\ell_p}{\left( \hat{P} \right) } $, for $ 1 \le p < \infty $ and 
$ {\ell_\infty}{\left( \hat{P} \right) } $ as the set of all sequences such that $ {P}{\left( \Delta^{\left( \tau \right) }\right) }- $ transforms are in the spaces $ {\ell_p} $ and $ \ell_\infty $ respectively, that is \
\[\begin{array}{l} {{\ell_p}\left( {\hat{P}}\right)  =\left\{\zeta=\left(\zeta_{k} \right)\in \omega :\hat{P} \left( \zeta \right)  \in \ell_p \right\}},  1 \le p < \infty \ and \\ 
{{\ell_ \infty}{\left( \hat{P} \right) } =\left\{\zeta=\left(\zeta_{k} \right)\in \omega : \hat{P} \left( \zeta \right)  \in \ell_ \infty \right\}}	.\end{array}\]  \
Alternatively above spaces can be written as :
 \begin{equation} \label{2.6}
 {\ell_p}{\left( \hat{P} \right) }= \left( {\ell_p}\right) _ {\left( \hat{P} \right) } \ and \ {\ell_\infty}{\left( \hat{P} \right) }= \left( {\ell_\infty}\right) _ {\left( \hat{P} \right) } . 
 \end{equation}
Now our introduced sequence spaces generalize the following sequence spaces as follows:\\
${\left( i\right)}$. For ${ \tau ={0} }$ and $ p = \infty $ , it reduces to $ {\ell_\infty}{\left( P \right) } $, studied by Polat [\cite{polat1}].\\
${\left( ii\right)}$. For ${ \tau ={1} }$ and $ p = \infty $ , it reduces to $ {\ell_\infty}{\left( P{\left( \Delta \right) } \right) } $,  where $ \ {\Delta}= {x_k}-{x_{k-1}} $, studied by Aydin and Polat [\cite{polat2}].\\
${\left( iii\right)} $. For ${ \tau ={m} }$ and $ p = \infty $, it reduces to $ {\ell_\infty}{\left( P{\left( \Delta^{\left( m \right) } \right) } \right) } $, \ and $\ \Delta^{\left( m \right) }=\sum_{i=0}^{m} {-1}^{j} \left(\begin{array}{l} {m} \\  {j} \end{array}\right) {x_{m-j}}$ studied by Aydin and Polat [\cite{polat3}]. \\
\noindent  Now with $\hat{P}$ - transform of $x={\left(x_{k}\right)}$ we define the sequence $y=\left(y_{n} \right)$ , $  n \in N $ as follows :
\begin{equation} \label{3.1} 
	\begin{array}{l} {y_{n} =\left( \hat{P} x \right)_{n} } =\sum _{k=0}^{n}\sum _{i=k}^{n}\left(-1\right)^{i-k}  \left(\begin{array}{l} {n} \\ {n-i} \end{array}\right)\frac{\Gamma \left({\tau} +1\right)}{\left(i-k\right)!\Gamma \left({\tau} -i+k+1\right)} {x_{k}} \end{array} .
\end{equation}
\section{Topological structure} 
\noindent This section deals with certain topological results of the spaces $ {\ell_p}{\left( \hat{P} \right) } $ and $ {\ell_\infty}{\left( \hat{P} \right) } $.  
\begin{thm}\label{thm2.1}
	{The sequence spaces $ {\ell_p}{\left( \hat{P} \right) } $ and $ {\ell_\infty}{\left( \hat{P} \right) } $ are $ BK- $ spaces with the norm defined by }
	\begin{equation}\label{2.4}
	\parallel x \parallel_{{\ell_p}{\left( \hat{P} \right) }} = \parallel{\left( \hat{P} x \right) }\parallel_{\ell_p} = {\left( \sum_{k} {\left|  {\left( \hat{P} \right) }_k {x} \right| ^p}\right)}^{1/p} \, 1 \le p < \infty \ 
	\end{equation}
and
\begin{equation}\label{2.5}
	\parallel x \parallel_{{\ell_\infty}{\left( \hat{P} \right) }} = \parallel{\left( \hat{P} x \right) }\parallel_{\ell_\infty} =  \sup_{k} {\left|  {\left( \hat{P} \right) }_k {x} \right|}  \
\end{equation}
\end{thm}
\begin{proof}
The proof of linearity is a routine verification and hence omitted. The spaces $ \ell_p $ and $ \ell_\infty $ are BK spaces with their natural norms. Since equation (2.3) holds and $ \hat{P} $ is a triangular matrix, Theorem 4.3.2 of Wilansky [\cite{wil}] implies the fact that $ {\ell_p}{\left( \hat{P} \right) } $ and $ {\ell_\infty}{\left( \hat{P} \right) } $ are BK spaces.	
	\end{proof}
\begin{thm}\label{thm2.2}
	The sequence spaces $ {\ell_p}{\left( \hat{P} \right) } $ and $ {\ell_\infty}{\left( \hat{P} \right) } $ are linearly isomorphic to $ \ell_p $ and $ \ell_\infty $ respectively, for $ 1 \le p \le \infty  $.
	\end{thm}
 \begin{proof}
 	\noindent Now define a mapping $T: {\ell_p}{\left( \hat{P} \right)}\to \ell_p   $ by    $x\to y=Tx$.
 	\noindent Clearly , $T$ is a linear transformation. If $Tx=\theta \ then \ x=\theta $, where $ \theta= \left( 0 \, , 0 \, , 0 \, ,. ..... \right)  $ ,  so  $T$ is one-one. \
 	Let $y\in \ell_p $ , define a sequence $x=\left(x_{k} \right)$  as
 	\begin{equation}\label{2.8}
 	x_{k} =\sum _{j=0}^{k} \sum _{i=j}^{k}\left(-1\right)^{k-j}  \left(\begin{array}{l} {i} \\ {i-j} \end{array}\right)\frac{\Gamma \left(-{\tau} +1\right)}{\left(k-i\right)!\Gamma \left(-{\tau} -k+i+1\right)}  y_{j} , (k \in N ) 
 	\end{equation}
 	Then \[ \parallel x \parallel_{{\ell_p}{\left( \hat{P} \right) }} 
 	= \left( {\sum_{k}} \left| \sum _{j=0}^{k}\sum _{i=j}^{k}\left(-1\right)^{i-j}  \left(\begin{array}{l} {k} \\ {k-i} \end{array}\right)\frac{\Gamma \left({\tau} +1\right)}{\left(i-j\right)!\Gamma \left({\tau} -i+j+1\right)} {x_{j}} \right|^{p} \right) ^{\frac{1}{p}} \] 
 	\[ = \left(  \mathop{\sum }\limits_{k} {\left| {\delta_{kj} \ y_{j}}\right|}^{p}\right) ^{\frac{1}{p}} =\left( {\mathop{\sum }\limits_{k}} \left| y_{k} \right|^{p} \right) ^{\frac{1}{p}} \ < \infty \]
 	\[ where \  {\delta_{kj}} = \left\lbrace  \begin{array}{l} { 1 , \qquad if\, k = j } \\ { 0 , \qquad if\, k \neq j } \end{array}\right. \]  
 	Thus, $ x \in  {\ell_p}{\left( \hat{P} \right) } $ and $T$ is a linear bijection and norm preserving. Hence the spaces $  {\ell_p}{\left( \hat{P} \right)}  $ and $ \ell_p $ are linearly isomorphic.
 	\noindent i.e.  $ {\ell_p}{\left( \hat{P} \right) } \cong \ell_p $. The proof for other spaces can be obtained in a similar manner.
 \end{proof}
\begin{thm}
	For $ 1\le p< \infty $ the space $ {\ell_p}{\left( \hat{P} \right) } $ is not a Hilbert space for $ p \neq 2 $ .
\end{thm}
\begin{proof}
	Here we prove that  $ {\ell_2}{\left( \hat{P} \right) } $ is the only Hilbert space in  $ {\ell_p}{\left( \hat{P} \right) } $ for $ 1\le p < \infty $. Since $ {\ell_p}{\left( \hat{P} \right) } $ is a BK-space for $ 1\le p< \infty $, so ${\ell_2}{\left( \hat{P}\right) } $ is also a BK-space with the norm \[  \parallel x \parallel_{{\ell_2}{\left( \hat{P} \right) }} \, = \parallel{\left( \hat{P} \right) }\parallel_{\ell_2} = {\left( \sum_{k} {\left|  {\left( \hat{P} \right) }_k {x} \right| ^2}\right)}^{1/2}\] \[={ \left\langle {{\left( \hat{P} \right) }x},{{\left( \hat{P} \right) }x}\right\rangle }^{1/2} \] \\
	for every $ x \in {\ell_2}{\left( \hat{P} \right) } $, we conclude that $ {\ell_2}{\left( \hat{P} \right) } $ is a Hilbert space. \\
	\noindent Let us consider two sequences $ u=(u_{k})=\left( 1 \, , \tau-1 \, , \frac{\tau(9-\tau)}{2}-6 \, , ....... \right)  $ and  $ v=(v_{k})=\left( 1 \, , \tau-4 \, , \frac{\tau(5-\tau)}{2} \, , ....... \right) $ then we have $ {{\left( \hat{P} \right) }u} = \left( 1 \, , 1 \, , 0 \, , 0 ...... \right) $ and $ {{\left( \hat{P} \right) }v}= \left( 1 \, , -1 \, ,  0 \, , 0  ...... \right)  $. Thus, it can be verify by the parallelogram law that is , \\
	L.H.S. = $ {\parallel u+v \parallel}^{2}_{{\ell_p}{\left( \hat{P} \right) }} + {\parallel u-v \parallel}^{2}_{{\ell_p}{\left( \hat{P} \right) }} = 8$ \\
	R.H.S. =$ 2\left( {\parallel u \parallel}^{2}_{{\ell_p}{\left( \hat{P} \right) }} + {\parallel v \parallel}^{2}_{{\ell_p}{\left( \hat{P} \right) }} \right)=4\left( 2^{2/p} \right) .\\ $
	For $ \,  p \neq 2 $, L.H.S. $\neq$ R.H.S. but for $ \,  p = 2 $, L.H.S. $=$ R.H.S. .\\
	 Hence  $ {\ell_p}{\left( \hat{P} \right) }$ is not a Hilbert space for $ p\neq 2 $ .
\end{proof}
	\begin{thm}
		The space  $ {\ell_p}{\left( \hat{P} \right) } $ for $ 1\le p< \infty $ is non-absolute type.
	\end{thm}
\begin{proof}
	Let consider a sequence $ w \in {\ell_{p}} $ and $ w=\left( 1 \, , 1 \, , 0 \, , 0 ...... \right) $. Then, \[ {\left( \hat{P} \right) }(w) = \left( 1 \, , \, 1-\tau \, , \,\frac{\tau(\tau-1)}{2} -{2\tau} \, , ....... \right)  \] and \[ {\left( \hat{P} \right) }(\left| w \right| ) = \left( 1 \, , \, 3-\tau \, , \,\frac{\tau(\tau-2)}{2} -{4\tau}+6 \, , ....... \right)  \]
	which implies that , 
	\[ \parallel w \parallel_{{\ell_p}{\left( \hat{P} \right) }} \neq \parallel {\left| w \right| } \parallel_{{\ell_p}{\left( \hat{P} \right) }} .\] 
\end{proof}
\section{The inclusion relation}
\noindent In this section we prove some inclusion relations concerning with the spaces $ {\ell_p}{\left( \hat{P} \right) } $ and $ {\ell_\infty}{\left( \hat{P} \right) } $.
\begin{thm}
	The inclusion $ {\ell_p} \subset {\ell_p}{\left( \hat{P} \right) }  $ strictly holds for $ 1 \le p < \infty $.
\end{thm}
\begin{proof}
	To prove the validity of the theorem, it is sufficient to show that there is a number $ M > 0 $ such that $ \parallel x \parallel_{{\ell_p}{\left( \hat{P} \right) }} = M . \parallel x \parallel_{\ell_p}  $ . 
	By applying H$\ddot{o}$lder's inequality on $ y_{k} $ we have
	\[\left| y_{k}\right| ^{p}= \left| \left( {\left( \hat{P} \right) } x \right)_{k}  \right| ^{p}  \] \[ = \left| \sum _{j=0}^{k}\sum _{i=j}^{k}\left(-1\right)^{i-j}  \left(\begin{array}{l} {k} \\ {k-i} \end{array}\right)\frac{\Gamma \left({\tau} +1\right)}{\left(i-j\right)!\Gamma \left({\tau} -i+j+1\right)} {x_{j}}\right|^{p} \]
	\[ \le \left( \sum _{j=0}^{k}\sum _{i=j}^{k} \left(\begin{array}{l} {k} \\ {k-i} \end{array}\right)\frac{\Gamma \left({\tau} +1\right)}{\left(i-j\right)!\Gamma \left({\tau} -i+j+1\right)} {\left| x_{j}\right|}^{p} \right) \] 
	\[ \, \, \times \left( \sum _{j=0}^{k}\sum _{i=j}^{k} \left(\begin{array}{l} {k} \\ {k-i} \end{array}\right)\frac{\Gamma \left({\tau} +1\right)}{\left(i-j\right)!\Gamma \left({\tau} -i+j+1\right)}  \right)^{p-1} \]
	\[ = \left( \sum _{j=0}^{k}\sum _{i=j}^{k} \left(\begin{array}{l} {k} \\ {k-i} \end{array}\right)\frac{\Gamma \left({\tau} +1\right)}{\left(i-j\right)!\Gamma \left({\tau} -i+j+1\right)} {\left| x_{j}\right|}^{p} \right) , \]
	Then,
	\[{\parallel x \parallel}^{p}_{{\ell_p}{\left( \hat{P} \right) }} = \left( {\sum_{k}} \left| \sum _{j=0}^{k}\sum _{i=j}^{k}\left(-1\right)^{i-j}  \left(\begin{array}{l} {k} \\ {k-i} \end{array}\right)\frac{\Gamma \left({\tau} +1\right)}{\left(i-j\right)!\Gamma \left({\tau} -i+j+1\right)} {x_{j}} \right|^{p} \right) \]
	\[ \le {\sum_{k}} \left( \sum _{j=0}^{k}\sum _{i=j}^{k} \left(\begin{array}{l} {k} \\ {k-i} \end{array}\right)\frac{\Gamma \left({\tau} +1\right)}{\left(i-j\right)!\Gamma \left({\tau} -i+j+1\right)} {\left| x_{j}\right|}^{p} \right) \]
	\[ \le {\sum_{k}} \left( \sum _{j=0}^{k} \left(\begin{array}{l} {k} \\ {k-j} \end{array}\right) {\left| x_{j}\right|}^{p} \right) =M.{\sum_{k}} \left( {\left| x_{k}\right|}^{p} \right) =M . {\parallel x \parallel}^{p}_{\ell_p} , \]
	where $ M=\sum _{j=0}^{k} \left(\begin{array}{l} {k} \\ {k-j} \end{array}\right). $
	This yields us that, ${\parallel x \parallel}^{p}_{{\ell_p}{\left( \hat{P} \right) }} \le M . {\parallel x \parallel}^{p}_{\ell_p}$ . Thus, $ {\ell_p} \subset {\ell_p}{\left( \hat{P} \right) } $ strictly holds for $ 1 \le p < \infty $. 
\end{proof}
 \section{Basis for the spaces}
\noindent In this section the Schauder basis \cite{madx3} for $ {\ell_p}{\left( \hat{P} \right) } $ is constructed. In theorem $ \left( 2.2 \right)  $ the spaces $  {\ell_p}{\left( \hat{P} \right) }  $ and $ \ell_p $ are linearly isomorphic, therefore the inverse of the basis $ {\left\lbrace e^{\left( k \right) } \right\rbrace }_{k \in N} $ of the space $ \ell_p $ forms the basis of $ {\ell_p}{\left( \hat{P} \right) } $. 
\begin{thm}\label{th4.1}
	let $ \mu_{k} {= { \left({\left( \hat{P} \right) }x\right)}_k} $. For $k \in N_{0}$ define $ {b^{\left( k\right)}}  = {\left\lbrace {{b_{n}}^{\left( k\right)}} \right\rbrace}_{n \in N_{0}} $ by
	\begin{multline}
		{\left\lbrace {{b_{n}}^{\left( k\right)}} \right\rbrace} =\left\{\begin{array}{l} { \sum _{j=k}^{n}\left(-1\right)^{n-k}  \left(\begin{array}{l} {j} \\  {j-k} \end{array}\right) \frac{\Gamma \left({-\tau} +1\right)}{\left(n-j\right)!\Gamma \left({-\tau}-n+j+1\right)}  , \qquad \qquad   if\, 0\le k\le n } \\ {0, \qquad \qquad \qquad \qquad \qquad \qquad \qquad \qquad  \qquad  \qquad  \qquad  if\, k>n} \end{array}\right. 
	\end{multline} 
	\noindent Then  $\left\lbrace {{b_{n}}^{\left( k\right)}} \right\rbrace$ is a Schauder basis for $ {\ell_p}{\left( \hat{P} \right) } $ so each $x\in {\ell_p}{\left( \hat{P} \right) } $ has a unique representation
	\[x=\sum _{k}\mu _{k} {{b_{n}}^{\left( k\right)}}, \ for \ each \  k \in N . \]
	where $ \mu _{k}={\left( \hat{P} \right) }_{k}. $
\end{thm}
\begin{proof}
	\noindent $(i)$  By the definition of ${\left( \hat{P} \right) }$ and ${{b_{n}}^{\left( k\right)}} $ , 
	\[ {\left( \hat{P} \right) }{{b_{n}}^{\left( k\right)}} ={e^{\left(k\right)}} \in \ell_p ,\] 
	Let  $x\in {\left( \hat{P} \right) } $, then
	$ x^{\left[r \right]} =\sum _{k=0}^{r} \mu_{k} {{b}^{\left( k\right)}}$ for an integer $ r \geq 0$.
	
	By applying $ {\left( \hat{P} \right) }$ we get \ $ {\left( \hat{P} \right) } x^{\left[r \right]} =\sum _{k=0}^{r} \mu _{k} {{b}^{\left( k\right)}}$
	\[=\sum _{k=0}^{s} {\mu_{k}}{e^{\left(k\right)}} = { \left(  {\left( \hat{P} \right) }x\right)}_k e^{\left(k\right)} \] and
	\begin{multline}
		{ {\left( \hat{P} \right) } \left( {x-x^{\left[r\right]}}\right)}_{s} =\left\{\begin{array}{l}  {0, \qquad\qquad \qquad \qquad if\, 0\le s \le r } \\ {{ \left(  {\left( \hat{P} \right) } x\right)}_s, \qquad \qquad  if\, s>r} \end{array}\right. ; for \ r,s \in {N_{0}}
	\end{multline} 
	\noindent For $\epsilon > 0$ there exist an integer $m_{o}$ s.t. 
	\[ \left( {\mathop{\sum }\limits_{r}}\left|{ \left(  {\left( \hat{P} \right) } x\right)}_{r}\right|^{p}\right) ^{\frac{1}{p}} < {\frac{\epsilon}{2}} \ for \ all \ r \geq {m_0} .\] 
	Hence \[ {\parallel \left( {x-x^{\left[r \right]}} \right) \parallel}_{ \left( {\ell_{p}} {\left( \hat{P} \right) } \right)} ={\mathop{\sum }\limits_{r }}{\left( \left|{ \left(  {\left( \hat{P} \right) } x\right)}_{r} \right|^{p}\right) }^{\frac{1}{p}} < {\frac{\epsilon}{2}} < {\epsilon} \ , for \ all \ r \geq {m_0}. \]
	\noindent Assume that $x=\sum _{k}\eta _{k}  {b}^{\left(k\right)} $. Since the linear mapping $T$ from $ \ell_p {\left( \hat{P} \right) } $ to $ \ell_p $ is continuous we have,
	\[{ \left({\left( \hat{P} \right) }x\right)}_{k}=\sum _{k}\eta _{k}  \left( {\left( \hat{P} \right) } {b}^{\left(k\right)} \right)_n \]
	\[ =\sum _{k}\eta _{k}  e^{\left(k\right)} = \eta _{n}  \ . \]
	This contradicts to our assumption that $ { \left({\left( \hat{P} \right) }x\right)}_{k} =\mu_{k}  $ for each $k\in {N_{0}}$. 
	Thus, the representation is unique.
\end{proof}
\section{${\alpha-}, \beta-$ and $\gamma-$ duals}

Here we determine  $\alpha -,\beta -$ and $\gamma - $ duals  of  $  {\ell_p}{\left( \hat{P} \right) } $  and $ {\ell_\infty}{\left( \hat{P} \right) }$.
We define \[S(X,Y)=\left\{u=\left(u\right)_{k} \in \omega :ux=\left(u_{k} x_{k} \right)\in Y,whenever\, x=\left(x_{k} \right)\in X\right\}\] as the multiplier sequence space for any two sequence spaces  X and Y.
Let $\alpha -,\beta -$ and $\gamma -$duals be denoted by 
\noindent $X^{\alpha } =S\left(X,l_{1} \right),\, X^{\beta } =S\left(X,cs\right),\,  X^{\gamma } =S\left(X,bs\right)$ respectively.
\noindent Throughout the collection of all finite subsets of $ \mathbb{N}$ is denoted by $\kappa $ . We consider $K\in \kappa$. Some lemmas needed in proving theorems.
\begin{lemma}\cite{gros}\label{lm5}
	\noindent Let $A=\left(a_{n,k} \right)$ be an infinite matrix. Then,
	\begin{enumerate}
	\item  $A\in \left( {\ell_p, \ell_1 } \right)$ iff 
	\begin{equation}
			{\mathop{\sup }\limits_{K\in \kappa }} \sum _{k}{\left|\sum _{k\in K}a_{nk} \right|} <\infty ,\qquad 1 < p \le \infty ;
		\end{equation}    
		\item  $A\in  \left( \ell_p, c \right)$ iff  
		\begin{equation} \label{4.5} 
			{\mathop{\lim } \limits_{n\to \infty }} a_{nk} \,  exists \, for \, all \, k \in N  ; 
		\end{equation} 
	\begin{equation} \label{4.6} 
		{\mathop{\sup }\limits_{n\in N}} \sum _{k}\left|a_{nk} \right|^{q} <\infty,  \, 1 < p < \infty .
	\end{equation} 
		\item  $A\in \left(  \ell_\infty, c \right)$ iff $ \left( 5.2 \right)  $ holds and  \begin{equation} \label{4.6} 
			{\mathop{\sup }\limits_{n\in N}} \sum _{k}\left|a_{nk} \right| <\infty ,
		\end{equation}
	\begin{equation}   
		{\mathop{\lim } \limits_{n\to \infty }}\sum _{k}\left|  a_{nk}- {\mathop{\lim } \limits_{n\to \infty }a_{nk}} \right| =0 .
		\end{equation}
		\item  $A\in \left( \ell_p, \ell_\infty \right)$ iff $ \left( 5.1 \right)  $ hplds with $ 1 < p\le \infty $.
	\end{enumerate}
\end{lemma}
\begin{thm}
\noindent Define the set $ {\tilde{D}_1}^{\left( \tau \right) } $ by 
\[ {\tilde{D}_1}^{\left( \tau \right) }=\left\lbrace a= \left( {a_{k}} \right) \in \omega : {\mathop{\sup }\limits_{n\in N}} \sum _{k}\left|\tilde{u}_{nk} \right|^{q} <\infty   \right\rbrace  \] Then,
\[ \left[ {\ell_p}{\left( \hat{P} \right) }\right]^{\alpha} ={\tilde{D}_1}^{\left( \tau \right) }, \, 1 < p \le \infty \].
\end{thm}
\begin{proof}
	\noindent  Let $a=\left(a_{k} \right)\in \omega $ and the sequence  $ \left( {x_k} \right)  $ defined in equation $ \left( 3.3 \right)  $ , then we derive that 
	\begin{equation}
		{a_n}{x_n}= \sum _{j=0}^{n} \sum _{i=j}^{n}\left(-1\right)^{n-j}  \left(\begin{array}{l} {i} \\ {i-j} \end{array}\right)\frac{\Gamma \left(-{\tau} +1\right)}{\left(n-i\right)!\Gamma \left(-{\tau} -n+i+1\right)}  y_{j} a_{n}  .
	\end{equation}
	\[= {\tilde{U}}_{n} \left( y\right), \, for \, each \,  n \in N .  \]
	where the matrix $ {\tilde{U}}_{n}= \left( {\tilde{u}_{nk}}\right)  $ is defined by
	\[ {\tilde{u}_{nk}}= \left\{\begin{array}{l} { \sum _{j=k}^{n}\left(-1\right)^{n-k}  \left(\begin{array}{l} {j} \\  {j-k} \end{array}\right) \frac{\Gamma \left({-\tau} +1\right)}{\left(n-j\right)!\Gamma \left({-\tau}-n+j+1\right)}   \qquad \qquad   if \, 0\le k\le n }, \\ {0 \qquad \qquad \qquad \qquad \qquad \qquad \qquad \qquad  \qquad  \qquad  \qquad \,  if \, k>n} \end{array}\right. \] for all $ n,k \in N $.
	Therefore we conclude that $ax=\left(a_{n} x_{n} \right)\in \ell_{1} $ when $x=\left(x_{k} \right) \in {\ell_p}{\left( \hat{P} \right) } $	iff $Uy \in \ell_{1}$ as $y=\left({y_k} \right) \in {\ell_{p}} $ .\\
	By lemma $\left( \ref{lm5}\right) $ we conclude that $\left\lbrace {\ell_p}{\left( \hat{P} \right) } \right\rbrace ^{\alpha}=\tilde{D}_{1}^{\left( \tau\right)} $.
\end{proof}
\begin{thm}
	\noindent Define the sets $ {\tilde{D}_2}^{\left( \tau \right) } $, $  {\tilde{D}_3}^{\left( \tau \right) } $, $ {\tilde{D}_4}^{\left( \tau \right) } $   by 
	\[ {\tilde{D}_2}^{\left( \tau \right) }= \left\lbrace a= \left( {a_{k}} \right) \in \omega : {\mathop{\sum }\limits_{k}} \left|\Delta^{\left( \tau \right) }{\left( {a_k}{d_k}\right) }\right|^{q} <\infty   \right\rbrace ; \]
	\[ {\tilde{D}_3}^{\left( \tau \right) }= \left\lbrace a= \left( {a_{k}} \right) \in \omega : {\mathop{\sup }\limits_{k}} \left| {a_k}{d_k}\right| <\infty   \right\rbrace ; \] and 
	\[{\tilde{D}_4}^{\left( \tau \right) }= \left\lbrace a= \left( {a_{k}} \right) \in \omega : {\mathop{\lim }\limits_{k}}  \left( {a_k}{d_k}\right)  =0  \right\rbrace ; \] where $ \Delta^{\left( \tau \right) }{\left( {a_n}{d_n}\right) }= {\sum_{i=0}^{n} \sum _{i=j}^{n}\left(-1\right)^{n-j}  \left(\begin{array}{l} {i} \\  {i-j} \end{array}\right) \frac{\Gamma \left({-\tau} +1\right)}{\left(n-i\right)!\Gamma \left({-\tau}-n+i+1\right)} {a_j} } $, then
	\[ \left[ {\ell_p}{\left( \hat{P} \right) }\right]^{\beta} ={\tilde{D}_2}^{\left( \tau \right) } \bigcap {\tilde{D}_3}^{\left( \tau \right) } ,  \, 1 < p \le \infty \] and 
	\[ \left[ {\ell_p}{\left( \hat{P} \right) }\right]^{\gamma} ={\tilde{D}_2}^{\left( \tau \right) } \bigcap {\tilde{D}_3}^{\left( \tau \right) } ,  \, 1 < p \le \infty \] .	
\end{thm}
\begin{proof}
\noindent  Let $a=\left(a_{k} \right)\in \omega $ and the sequence  $ \left( {x_k} \right)  $ defined in equation $ \left( 3.3 \right)  $ , consider 
\[ \sum_{k=0}^{n}{a_k}{x_k}= \sum _{k=0}^{n} {a_k} \left[ \sum_{i=0}^{k} \sum _{j=i}^{k}\left(-1\right)^{k-j}  \left(\begin{array}{l} {i} \\ {i-j} \end{array}\right)\frac{\Gamma \left(-{\tau} +1\right)}{\left(k-i\right)!\Gamma \left(-{\tau} -k+i+1\right)}  y_{j}\right]  \]	
\[ \sum _{k=0}^{n} \left[\sum_{i=0}^{k} \sum _{j=i}^{k}\left(-1\right)^{k-j}  \left(\begin{array}{l} {i} \\ {i-j} \end{array}\right)\frac{\Gamma \left(-{\tau} +1\right)}{\left(k-i\right)!\Gamma \left(-{\tau} -k+i+1\right)}{a_j} \right] y_{k} \]
\[ =  \sum _{k=0}^{n}\Delta^{\left( \tau \right) }{\left( {a_k}{d_k}\right) }{y_k} \]
\begin{equation}
	=\tilde{V}_{n}{\left( y \right) } , \, for \, each \, n \in N .
\end{equation} 
where the matrix $ \tilde{V}_{n} = {\left( {\tilde{v}_{nk}}\right) } $ defined by
\[ {\tilde{u}_{nk}}= \left\{\begin{array}{l} {\Delta^{\left( \tau \right) }{\left( {a_k}{d_k}\right) }   \qquad \qquad   if \, 0\le k< n }, \\
	{1 \qquad \qquad \qquad \qquad  \qquad  \,  if \, k=n} \\ {0 \qquad \qquad \qquad \qquad \qquad \,  if \, k>n} \end{array}\right. \] for all $ n,k \in N $.
Therefore we deduce that $ax=\left(a_{n} x_{n} \right)\in cs $ whenever $x=\left(x_{k} \right) \in {\ell_p}{\left( \hat{P} \right) } $	iff $\tilde{V}y \in c$ as $y=\left({y_k} \right) \in {\ell_{p}}$ .\
By using lemma $\left( \ref{lm5}\right) $ we conclude that $\left\lbrace {\ell_p}{\left( \hat{P} \right) } \right\rbrace ^{\beta}= \tilde{D}_{2}^{\left( \tau\right)}  \bigcap {\tilde{D}_3}^{\left( \tau \right) } ,  \, 1 < p \le \infty $.
Hence the theorem proved and the duals of other spaces can be obtained in a similar manner using lemma $\left( \ref{lm5}\right) $ .
\end{proof}
\section{Conclusion}
\noindent In this paper, certain results on Pascal difference sequence spaces of order m,$ (m \in \mathbb{N})$ have been extended to Pascal difference sequence spaces of fractional order $ \tau $. Also we introduced certain topological properties of new difference sequence spaces. For future work we will study certain geometrical properties and matrix transformation of the spaces.


\begin{thebibliography}{}
	\bibitem{wil} A. Wilansky , "Summability Through Functional Analysis", North Holland Mathematics Studies, vol. 85, \textit{Elsevier} , Amsterdam , 1984		
		
	\bibitem{alta2} B. Altay and F. Ba\c{s}ar, "On some Euler sequence spaces of nonabsolute type", \textit{Ukr. Math. J.} 57,1-17, 2005.
	
	\bibitem{alta3} B. Altay, F. Ba\c{s}ar and M. Mursaleen, "On the Euler sequence spaces which include the spaces lp and l$\mathrm{\infty}$, I" \textit{Inf. Sci.} 176, 1450-1462, 2006.
	
	\bibitem{alta4} B. Altay and F. Ba\c{s}ar, "On the paranormed Riesz sequence space of non-absolute type", \textit{Southeast Asian Bulletin of Mathematics.}vol. 26, pp. 701–7152002.
	
	\bibitem{alta5} B. Altay and F. Ba\c{ş}ar, "Some paranormed Riesz sequence spaces of non-absolute type", \textit{Southeast
		Asian Bull. Math.} 30, 591–608, 2006.
	
	\bibitem{alta1} B.Altay and H. Polat, "On some new Euler difference sequence spaces", \textit{Southeast Asian Bull. Math.} 30, 209-220, 2006.
	
	\bibitem{polat1} H. Polat, Some new Pascal sequence spaces, Fundam. J. Math. Appl., 1(1) (2018), 61-68
	
	\bibitem{bali6} P. Baliarsingh, "Some new difference sequence spaces of fractional order and their dual spaces", \textit{Appl. Math. Comput.} 219, 9737-9742, 2013.
	
	\bibitem{bali4} P. Baliarsingh and S. Dutta,  "A unifying approach to the difference operators and their applications", \textit{Bol. Soc. Paran.} 33, 49-57, 2015.
	
	\bibitem{bali5} P. Baliarsingh and S. Dutta, "On the classes of fractional order difference sequence spaces and their matrix transformations", \textit{Appl. Math. Comput.} 250, 665-674, 2015.
	
	\bibitem{bali7} P. Baliarsingh and S. Dutta, "On an explicit formula for inverse of triangular matrices", \textit{J. Egypt. Math. Soc.} 23, 297-302, 2015.
	
	\bibitem{bara3} F. Ba\c{s}ar and N.L. Braha, "Euler- Cesaro Difference Spaces of Bounded, Convergent and Null Sequences" \textit{Tamkang J. Math.} 47, 4, 405–420, 2016.
	
	\bibitem{basa2} M. Başarir, "On the generalized Riesz B-difference sequence spaces", \textit{Filomat} 24, 4. 35–52, 2010.
	
	\bibitem{basa1} M.Başarir and M. Öztürk, "On the Riesz difference sequence space", \textit{Rend. Circ. Mat. Palermo}, 2,57, 377–389, 2008.
	
	
	\bibitem{basa4} M. Başarir and M. Öztürk, "On some generalized Bm-difference Riesz sequence spaces and uniformopial property.\", \textit{J. Inequal. Appl.}, Article ID 485730, 17 pages, 2011.
	
	\bibitem{braw} R. Brawer, Potenzen der Pascal matrix und eine Identitat der Kombinatorik, Elem. der Math., 45 (1990), 107-110.
	
	\bibitem{polat2} S. Aydın and H. Polat, “Difference sequence spaces derived by using Pascal transform”, Fundam. J. Math. Appl. 2:1 (2019),
	56–62.
	
	\bibitem{polat3}S. Aydın and H. Polat, “Some Pascal spaces of difference sequence spaces of order m”, Conf. Proc. Sci. Technol. 2:1 (2019),
	97–103.
	
	\bibitem{dutt11} S. Dutta and P. Baliarsingh, "A note on paranormed difference sequence spaces of fractional order and their matrix transformations" \textit{J. Egypt. Math. Soc.} 22, 249-253, 2014.
	
	\bibitem{dutt10} S. Dutta and P. Baliarsingh, "On some Toeplitz matrices and their inversion", \textit{J. Egypt. Math. Soc.} 22, 420-423, 2014.
	
	\bibitem{demri2} H.B. Ellidokuzoglu and S. Demiriz, "Euler-Riesz difference sequence spaces", \textit{Turkish J. Math. Comput. Sci.}, 7, 63–72, 2017.
	
	
	\bibitem{et13} M. Et and R. \c{C}olak, "On generalized difference sequence spaces", \textit{Soochow J. Math.}, 21, 377-386, 1995.
	
	
	\bibitem{basa6} M. Et and M. Basarir, "On some new generalized difference sequence spaces", \textit{Periodica Math. Hungar.} 35: 169–175, 1997. 
	
	\bibitem{gros} K.G. Grosse-Erdmann,1993. "Matrix transformation between the sequence spaces of Maddox", \textit{J. Math. Anal. Appl.}, 180, 223–238, 1993.
	
	\bibitem{kada14} U. Kadak and P. Baliarsingh, "On certain Euler difference sequence spaces of fractional order and related dual properties", \textit{J. Nonlinear Sci. Appl.} 8: 997-1004, 2015.
	
	\bibitem{kizm15} H. Kizmaz, "On certain sequence spaces", \textit{Can. Math. Bull.} 24: 169-176, 1981.
	
	\bibitem{madx1} I.J. Maddox, "Spaces of strongly summable sequences", \textit{Q. J. Math.}, 18, 345–355, 1967.
	
	\bibitem{madx2} I.J. Maddox, "Paranormed sequence spaces generated by infinite matrices" \textit{Mathematical Proceedings of the Cambridge Philosophical Society}, 64, 335–340, 1968.
	
	\bibitem{madx3} I.J. Maddox, "Elements of Functional Analysis", \textit{2nd ed., Cambridge University Press, Cambridge}, 1988.
	
	\bibitem{murs13} M. Mursaleen and K. Noman, "On Some new sequence spaces of non-absolute type related to the spaces $l_p$ and $l_\infty$- I", \textit{Filomat.}, 25, 33-51, 2, 2011.
	
	\bibitem{murs2011} M. Mursaleen and K. Noman, "On Some new sequence spaces of non-absolute type related to the spaces $l_p$ and $l_\infty$- II" \textit{Mathematical Communications}, 16, 383–398, 2011.
	
	\bibitem{mur1} M. Mursaleen, "Generalized spaces of difference sequences" \textit{J. Math. Anal. Appl.}, 203, 738–745, 1996.
	
	\bibitem{pola16} H. Polat and F. Başar, "Some Euler spaces of difference sequences of order m" \textit{Acta Math. Sci.}, 27, 254-266, 2007.
	
	\bibitem{ss1} S. Singh and S. Dutta, "On certain generalized  mth order geometric difference sequence spaces" \textit{FJMS.}, 116, 1. 83-100, 2019.
	
	\bibitem{ss2} S. Singh, S. Dutta, D. Dash, and R. Sharma, "Strongly summable Fibonacci Difference Geometric Sequences derfined by Orlicz functions" \textit{GANITA.}, 71, 2, 99-109, 2021.
	
	\bibitem{sonm} A. Sönmez and F. Başar "Generalized Difference Spaces of Non-Absolute Type of Convergent and Null Sequences", \textit{Abstr. Appl. Anal.}, Article ID 435076, 20 pages, 2012.
	
	\bibitem{taja1} T. Yaying, "Paranormed Riesz difference sequence spaces of fractional order", \textit{Kragujevac J. Math.}, 46, 175-191, 2022.

	\bibitem{aydin1} Aydın, C. and Ba¸sar, F., Some new sequence spaces which include the spaces $ \ell_{p} $ and $ \ell_\infty $. Demonstratio Math. 38	(2005), no. 3, 641-656
	\end{thebibliography}
\end{document}